\documentclass[a4paper,10pt]{amsart}

\usepackage{a4wide}
\usepackage[latin1]{inputenc} 
\usepackage[T1]{fontenc}

\usepackage{amsfonts}
\usepackage{amsmath,amssymb,amsthm,stmaryrd}

\usepackage{mathrsfs}
\usepackage{esint}

\usepackage{fancyhdr}
\usepackage{verbatim}

\usepackage{graphicx}
\usepackage{ifpdf}
\ifpdf
\fi 

\newcommand{\esssup}[0]{\operatornamewithlimits{ess\,sup}}

\providecommand{\abs}[1]{\lvert#1\rvert} 

\newcommand{\loc}[0]{\operatorname{loc}}

\newcommand{\R}{\mathbb{R}}

\newcommand{\Z}{\mathbb{Z}}

\swapnumbers
\numberwithin{equation}{section}

\makeatletter
  \let\c@equation\c@subsection
\makeatother

\theoremstyle{plain}
\newtheorem{theorem}[subsection]{Theorem}
\newtheorem{lemma}[subsection]{Lemma}

\newtheorem{proposition}[subsection]{Proposition}

\newtheorem{mainlemma}[subsection]{Main lemma} 

\theoremstyle{definition}
\newtheorem{definition}[subsection]{Definition}

\theoremstyle{remark}
\newtheorem{remark}[subsection]{Remark}
\newtheorem{remarks}[subsection]{Remarks}

\makeatletter
\@namedef{subjclassname@2010}{%
  \textup{2010} Mathematics Subject Classification}
\makeatother

\title[Sharp weighted bounds for fractional integral operators]{Sharp weighted bounds for fractional integral operators in a space of homogeneous type}

\author{Anna Kairema}

\address{Department of Mathematics and Statistics, P.O.B. 68 (Gustaf H\"allstr\"omin katu 2), FI-00014 University of Helsinki, Finland}
\email{anna.kairema@helsinki.fi}

\thanks{The author is supported by the Academy of Finland, grant 133264.}

\begin{document}

\begin{abstract}
We consider a version of M. Riesz fractional integral operator on a space of homogeneous type and show an analogue of the well-known Hardy--Littlewood--Sobolev theorem in this context. In our main result, we investigate the dependence of the operator norm on weighted spaces on the weight constant, and find the relationship between these two quantities. It it shown that the estimate obtained is sharp in any given space of homogeneous type with infinitely many points. Our result generalizes the recent Euclidean result by Lacey, Moen, P{\'e}rez and Torres \cite{LMPT10}.
\end{abstract}

\maketitle

\section{Introduction}

In the Euclidean space $\R^n$, the fractional integral operator $I_\alpha$ of order $0<\alpha<n$, the Riesz potential, is defined by
\[I_\alpha f(x)=\int_{\R^n}\frac{f(y)}{\abs{x-y}^{n-\alpha}}dy. \]
At a formal level, the limit $\alpha\to 0$ corresponds to the Calderón--Zygmund case, and for $\alpha>0$ one deals with a positive operator. This classical potential has been studied in depth by several authors. By the well-known Hardy--Littlewood--Sobolev theorem, $I_\alpha$ is a bounded operator from $L^p(\R^n)$ to $L^q(\R^n)$ if and only if $p>1$ and $1/p-1/q=\alpha/n$. The qualitative one weight problem was solved in the early 1970's in the work of B. Muckenhoupt and R. L. Wheeden \cite{MW74} giving a characterization of weights $w$ for which $I_\alpha\colon L^p(\R^n,w^pdx)\to L^q(\R^n,w^qdx)$ is bounded: For $1<p<n/\alpha$ and $1/p-1/q=\alpha/n$, the inequality
\[\|I_\alpha f \|_{L^q(w^q)} \leq C\|f \|_{L^p(w^p)}\]
holds if and only if
\[[w]_{A_{p,q}}:=\sup_{Q\text{ a cube}}\left(\frac{1}{\abs{Q}}\int_{Q}w^q\,dx\right)\left(\frac{1}{\abs{Q}}\int_{Q}w^{-p'}\right)^{q/p'}<\infty .\]
Nevertheless, the precise dependence of the operator norm on the weight constant $[w]_{A_{p,q}}$ was not considered in detail until the very recent times. The original interest in sharp estimates was motivated by applications in other areas of analysis. In the last decade, the interest in understanding such quantitative questions has been a general trend in the study of integral operators. 

The central $A_2$ conjecture for Calderón--Zygmund operators was only recently solved by T. Hyt\"onen \cite{TH}. Sharp estimates for $I_\alpha$ were obtained somewhat earlier by Lacey et al. \cite{LMPT10}: For $1<p<n/\alpha$ and $1/p-1/q=\alpha/n$,
\begin{equation}\label{eq:main_inequality}
\|I_\alpha\|_{L^p(w^p)\to L^q(w^q)}\lesssim [w]_{A_{p,q}}^{\left(1-\frac{\alpha}{n}\right)\max\{1,\frac{p'}{q}\}}, 
\end{equation}
and the estimate is sharp. The main result in the present paper is the extension of this estimate into general spaces of homogeneous type. 

\subsection{Fractional integral operators in metric measure spaces}
Fractional integrals over quasi-metric measure spaces $(X,\rho,\mu)$ are known to be considered in different forms. In this paper, one possible variant is studied. In order to place our investigations into a larger picture, we shall briefly comment on the different notations and 
the relationship between the operator chosen here and some other types of fractional integrals studied by other authors elsewhere.

First, one common and widely studied notation; see e.g. the book \cite{EKM} and the paper \cite{G06}, is given by the formula
\[I^s f(x):= \int_{X}\frac{f(y)\,d\mu(y)}{\rho(x,y)^{s}},\quad s>0,\]
and it has been studied in both the doubling \cite{GattoVagi1, GSV96, KK89} and non-doubling \cite{GCG03, KM01, KM05} setting. 
Operators $I^s$ are better suited for non-doubling measure spaces $(X,\mu)$ with the \textit{upper Ahlfors regularity condition} that for some $n>0$,
\begin{equation}\label{eq:AhlReg}
\mu(B(x,r))\leq C_1r^n
\end{equation}
where $C_1>0$ does not depend on $x\in X$ and $r>0$.
Indeed, by the analogue of the Hardy--Littlewood--Sobolev Theorem; see e.g. \cite[Theorem 1]{G06}, $I^s$ is a bounded operator from $L^p(X,\mu)$ to $L^q(X,\mu)$ for $1<p<q<\infty$ if and only if $\mu$ satisfies \eqref{eq:AhlReg}, $s=n-\alpha$ with $0<\alpha<n$, and $1/p-1/q=\alpha/n$.
Second, another types of fractional integrals are given by
\[\mathfrak{T}^\alpha f(x):= \int_{X}\frac{\rho(x,y)^\alpha}{\mu(B(x,\rho(x,y)))}f(y) \,d\mu(y), \quad \alpha>0. \]
These operators are studied e.g. in the book \cite{GGKK98} and the paper \cite{AS08}, and they are better adjusted for and commonly studied in measure spaces $(X,\mu)$ with the \textit{lower Ahlfors regularity condition} that for some $n>0$,
\begin{equation}\label{eq:AhlReg2}
\mu(B(x,r))\geq C_2r^n,
\end{equation}
where $C_2>0$ does not depend on $x\in X$ and $r>0$.
We mention that, for example, all \textit{doubling measures} (see Section \ref{subsec:setup} for a definition) satisfy the lower Alhfors regularity condition with $n=\log_2 C_\mu$  and $C_2=C_\mu$, and for all $x\in \Omega$ and $0<r\leq \ell$ where $\ell<\infty$ is any fixed number and $\Omega\subseteq X$ is any open set with the property that $\inf_{x\in \Omega}\mu(B(x,\ell))>0$.

In the present paper we study fractional integrals of the type
\begin{equation}\label{def:operator}
T_\gamma f(x):= \int_{X}\frac{f(y)\,d\mu(y)}{\mu(B(x,\rho(x,y)))^{1-\gamma}},\quad 0<\gamma<1 ,
\end{equation}
in a space of homogeneous type. These operators have been studied e.g. in \cite{BC94, EKM, GattoVagi1, KS09} in the same context. 
Obviously, \eqref{eq:AhlReg} implies that 
\[I^sf(x)=I^{n-\alpha}f(x)\leq C_1
\left\{ \begin{array}{l}  T_\gamma f(x)\\
\mathfrak{T}^\alpha f(x)
\end{array} \right. 
\quad\text{and} \quad T_\gamma f(x) \leq C_1 \mathfrak{T}^\alpha f(x)\]
for $f\geq 0$ and $\gamma :=\alpha/n$. Similarly, \eqref{eq:AhlReg2} implies 
\[\mathfrak{T}^\alpha f(x)\leq \frac{1}{C_2}
\left\{ \begin{array}{l}
I^{n-\alpha}f(x)\\
 T_\gamma f(x) 
 \end{array}\right.
 \quad\text{and}\quad
 T_\gamma f(x)  \leq \frac{1}{C_2}\;I^{n-\alpha}f(x)\]
for $f\geq 0$ and $\gamma :=\alpha/n$. If $X$ has a constant dimension in the sense that $\mu$ satisfies both the regularity conditions \eqref{eq:AhlReg} and \eqref{eq:AhlReg2}, then all the three variants of fractional integrals mentioned are equivalent. Accordingly, our results apply to all of them. In particular, in the usual Euclidean space $\R^n$ with the Lebesgue measure, all the three operators reduce to the classical Riesz potentials. 

Finally, we mention that also some further types of fractional integrals have been considered elsewhere; see \cite[Chapter 6]{EKM}.

\bigskip

\noindent
The paper is organized as follows. First, we show an analogue of the Euclidean Hardy--Littlewood--Sobolev theorem in our context. This preliminary unweighted result motivates us to restrict our considerations only to exponents which satisfy the identity $1/p-1/q=\gamma$. Second, in our main result we show that the estimate \eqref{eq:main_inequality} holds for $T_\gamma$ defined in \eqref{def:operator} and $1/p-1/q=\gamma$. Finally, we show that this estimate is sharp in any given space $(X,\mu)$ with infinitely many points. We do this by showing that any such space supports functions which, at least locally, behave sufficiently similarly to the basic power functions $\abs{x}^{-\alpha}$ on the Euclidean space.

\textit{Acknowledgements:} The author has been supported by the Academy of Finland, project 133264. The paper is part of the author's PhD thesis written under the supervision of Associate professor Tuomas Hyt\"onen. The author wishes to express her gratitude for the organizers of the 9th International Conference on Harmonic Analysis and PDE held in El Escorial, Spain, on June 11--15, 2012. Communications with other participants during the conference brought into the author's attention the other type fractional integrals, different from the one considered here, that have been studied elsewhere in the context of metric measure spaces. The author also appreciates Theresa C. Anderson for her useful questions and remarks on the manuscript.

\section{Preliminaries}\label{sec:preliminaries}

\subsection{Set-up}\label{subsec:setup}
Our set-up is a space of homogeneous type $(X,\rho,\mu)$ with a quasi-metric $\rho$ and a doubling measure $\mu$. By a quasi-metric we mean a mapping that satisfies the axioms of a metric except for the triangle inequality, which is assumed in the weaker form
\[\rho(x,y)\leq A_0(\rho(x,z)+\rho(z,y))\] 
with a quasi-metric constant $A_0\geq 1$.
As usual, for a ball $B=B(x,r):=\{y\in X\colon \rho(x,y)<r\}$ and $a>0$, the notation $aB:=B(x,ar)$ stands for the concentric dilation of $B$. By a doubling measure we mean a positive Borel-measure $\mu$ defined on a $\sigma$-algebra of subsets that contains the quasi-metric balls which has the doubling property that there exists a constant $C\geq 1$ such that 
\begin{equation}\label{def:doubling}
0<\mu(B(x,2r))\leq C\mu(B(x,r))<\infty \quad\text{for all $x\in X$, $r>0$}.
\end{equation}
The smallest constant satisfying \eqref{def:doubling} is denoted by $C_\mu$ and referred to as the doubling constant.

We recall the following properties of a doubling measure.

\begin{lemma}\label{lem:doubl;poinmass}
Let $(X,\rho ,\mu)$ be a space of homogeneous type. Then the following is true.

(i) For $x\in X$, $\mu(\{x \})>0$ if and only if there exists $\varepsilon>0$ such that $\{x\}=B(x,\varepsilon)$.

(ii) The set $\{x\in X\colon \mu(\{x\})>0\}$ of atoms is at most countable.
In particular, if $\mu(\{x \})\geq \delta >0$ for all $x\in X$, then $X$ is countable.

(iii) $\mu(X)<\infty$ if and only if $X=B(x,R)$ for some $x\in X$ and $R<\infty$.
\end{lemma}
We mention that, in fact, the property (ii) does not depend on the doubling property; only $\sigma$-finiteness is needed for this. 
We further recall the following well-known result.

\begin{lemma}\label{prop;balls;measures}
Suppose $\mu$ is a doubling measure. Then for every $x\in X$ and $0<r\leq R$ we have
\[\frac{\mu(B(x,R))}{\mu (B(x,r))}\leq C_\mu \left(\frac{R}{r}\right)^{c_\mu} \] 
where $c_\mu=\log_2C_\mu$.
\end{lemma}

\subsection{A space of homogeneous type with infinitely many points}
Some of our results, most importantly the example in Section~\ref{sec:example}, require that $X$ is sufficiently non-trivial in that it contains infinitely many points. 

\begin{lemma}\label{lem:property}
Let $(X,\rho,\mu)$ be a space of homogeneous type. The property $\# X=\infty$ is equivalent to the property that
\begin{equation}\label{def:property}
\text{for any $N>0$ there exist balls $B_0$ and $B_1$ such that } 
\mu(B_1)>N\mu(B_0).
\end{equation}
\end{lemma}
\begin{proof}
It is clear that \eqref{def:property} implies $\# X=\infty$. Indeed, if $X=\{ x_1,\ldots ,x_k\}$ and thereby, $\mu(\{x_i\})\geq \delta >0$ and $\mu(X)=\sum_{i}\mu(\{x_i\})<\infty$, then $X$ can not have the property \eqref{def:property}.

We are left to show that the property $\# X=\infty$ implies \eqref{def:property}. Let $N>0$. First suppose that there exists $x_0\in X$ with the property that $\mu(\{x_0\})=0$. Choose $B_1=B(x_0,1)$ and $B_0=B(x_0,\varepsilon)$ where $\varepsilon>0$ is small so that $\mu(B(x_0,\varepsilon))<1/(N\mu(B_1))$. Then $\mu(B_1)/\mu(B_0)>N$.

Then assume that $\mu(\{x\})>0$ for all $x\in X$ and thus, $\{x\}=B(x,\varepsilon)$ for some $\varepsilon=\varepsilon(x)>0$. 

\textit{Case 1.} Suppose that there exists $\delta>0$ such that $\mu(\{x\})\geq \delta$ for all $x$. Given $N>0$, choose $B_0:=\{x_0\}$ and $B_1:=B(x_0,M)$ where $M$ is large so that $x_i\in B_1$ for at least $N\mu(B_0)/\delta +1$ different $i$ (since $\# X=\infty$, such an $M$ exists). Then
\[\frac{\mu(B_1)}{\mu(B_0)}\geq \frac{(N\mu(B_0)/\delta +1)\delta}{\mu(B_0)}> N.\]

\textit{Case 2.} Then assume that $\mu(\{x_i\})\to 0$ as $i\to \infty$. Given $N>0$, choose $B_1:=\{x_1\}$ and $B_0:=\{x_i\}$ where $i$ is large so that $\mu(\{x_i\})< \mu(B_1)/N$. Then $\mu(B_1)/\mu(B_0)>N$.
\end{proof}

\begin{remarks}
\textit{1.} Suppose that $X$ has the property \eqref{def:property}. Then the balls $B_0$ and $B_1$ in \eqref{def:property} may be assumed to have a mutual centre point. Indeed, suppose $B_i=B(x_i,r_i), i=0,1$, and let $R\geq r_1$ be large so that $B_1\subseteq \tilde{B}_1:=B(x_0,R)$. Then $\mu(\tilde{B}_1)\geq \mu(B_1)>N\mu(B_0)$ where $\tilde{B}_1$ and $B_0$ have a mutual centre point.

\textit{2.} Lemma~\ref{lem:property} in particular implies that if $X$ has infinitely many points, then at least one of the following two conditions is satisfied by balls $B$ in $X$: (A) $\mu(B)$ can have arbitrarily small values; (B) $\mu(B)$ can have arbitrarily large values which is equivalent to $\mu(X)=\infty$. Conversely, both (A) and (B) imply \eqref{def:property} and thereby also that $\# X=\infty$. This observation leads to the three categories of spaces listed in Lemma~\ref{lem:three_choices} below.

\textit{3.} The basic intuition behind \eqref{def:property} is that the quantities $\mu(\{x\})$ are, in some sense, vanishing. Indeed, the condition \eqref{def:property} entails points $x\in X$ with the property that $\mu(\{x\})\ll \mu(B(x,R))$ for some $R>0$. Thus, by working on this larger scale (we informally re-scale the measure so that $\mu(B(x,R))\approx 1$), the measure of the singleton $\{x\}$ becomes negligible. This observation will help us in Section \ref{sec:example} to construct functions that, at least locally, behave similarly to the power functions $\abs{x}^{-\alpha}$ on the Euclidean spaces.
\end{remarks}

We show that every space of homogeneous type with infinitely many points belongs to one of the three categories listed in Lemma~\ref{lem:three_choices} below. The most basic examples of such three categories are provided by $([0,1],dx), (\Z,\mu)$ and $(\R,dx)$, respectively, where $dx$ denotes the one-dimensional Lebesgue measure and $\mu$ is the counting measure.
\begin{lemma}\label{lem:three_choices}
Suppose that $(X,\mu)$ is a space of homogeneous type and $\# X=\infty$. Then precisely one of the following is satisfied:
\begin{enumerate}
\item[\textit{1)}] $\mu(X)<\infty$;
\item[\textit{2)}] $X$ is countably infinite and $\mu(\{x\})\geq \delta>0$ for all $x\in X$;
\item[\textit{3)}] $\mu(B)$ can have arbitrarily small and large values with balls $B$.
\end{enumerate}
\end{lemma}
\begin{proof}
First note that if $\mu(X)<\infty$, then \textit{3} can not be satisfied. The property $\# X=\infty$ implies \eqref{def:property}, and thereby we must have that $\mu(B)$ can have arbitrarily small values so that also \textit{2} fails. Also note that \textit{2} and \textit{3} are mutually exclusive properties.

Then suppose $\mu(X)=\infty$. Thus, $\mu(B)$ can have arbitrarily large values. If \textit{3} is not satisfied, then $\mu(B)\geq \delta>0$ for all balls which implies that $\mu(\{x\})\geq \delta>0$ for all $x\in X$. Hence $X$ is countable by Lemma~\ref{lem:doubl;poinmass}(ii), and \textit{2} is satisfied.
\end{proof}

\begin{remark}\label{rem:organized}
We record the following easy observations:

\textit{1.} Suppose that $X$ belongs to the category 1 of Lemma \ref{lem:three_choices}. Thus, $\mu(X)<\infty$. H\"older's inequality implies that $L^{p_1}(X)\subseteq L^{p_0}(X)$ for all $p_1\geq p_0\geq 1$. 

\textit{2.} Suppose that $X$ belongs to the category 2 of Lemma \ref{lem:three_choices} so that $\mu(\{x\})\geq \delta>0$ for all $x\in X$. Then $X$ is a countable set. Now we have that that $L^{p_0}(X)\subseteq L^{p_1}(X)$ for all $p_1\geq p_0\geq 1$. Indeed,
\begin{align*}
\|f \|_{L^{p_1}} &=\left(\sum_{x\in X}\abs{f(x)}^{p_1}\mu(\{x\})\right)^{1/p_1} =\left(\sum_{x\in X}\abs{f(x)}^{p_1}\mu(\{x\})\right)^{p_0/p_1\cdot 1/p_0}\\
& \leq \left(\sum_{x\in X}\abs{f(x)}^{p_0}\mu(\{x\})^{p_0/p_1}\right)^{1/p_0} \quad\text{since }p_0\leq p_1,\\
& =\delta^{1/p_1-1/p_0}\left(\sum_{x\in X}\abs{f(x)}^{p_0}\mu(\{x\})\left(\frac{\mu(\{x\})}{\delta}\right)^{p_0/p_1-1}\right)^{1/p_0}\\
&\leq \delta^{1/p_1-1/p_0}\left(\sum_{x\in X}\abs{f(x)}^{p_0}\mu(\{x\})\right)^{1/p_0}
=\delta^{1/p_1-1/p_0}\|f \|_{L^{p_0}} 
\end{align*}
since $p_0/p_1-1\leq 0$ and $\mu(\{x\})/\delta\geq 1$.

Thus, the two properties $\mu(X)<\infty$ and $\mu(\{x\})\geq \delta>0$ for all $x\in X$ organize the $L^p(X,\mu)$ spaces in mutually reversed order.
\end{remark}

\subsection{Weight classes of interest}
We recall definitions and some easy results concerning the classes of weights relevant in our investigations. A non-negative locally integrable function $w$ is a \textit{weight}. A weight defines a measure (denoted by the same symbol) $w(E):=\int_Ewd\mu$. We say that a weight $w$ belongs to the $A_p$ class for $1<p<\infty$ if it satisfies the condition
\[[w]_{A_p}:=\sup_{B}\left(\frac{1}{\mu(B)}\int_{B}wd\mu\right)\left(\frac{1}{\mu(B)}\int_{B}w^{-\frac{1}{p-1}}d\mu\right)^{p-1}<\infty ,\]
where the supremum is over all balls $B$ in $X$. The quantity $[w]_{A_p}\geq 1$ is then called the $A_p$ constant of the weight $w$. 
A weight $w$ is said to belong to the $A_1$ class if 
\[Mw\leq Cw \quad \text{a.e} ,\]
where $M$ is the Hardy--Littlewood maximal operator defined by
\[Mf(x)=\sup_{B}\frac{\chi_B(x)}{\mu(B)}\int_{B}\abs{f}d\mu\quad\text{for }f\in L^1_{\loc}(X). \]
The smallest possible constant $C$ is then called the $A_1$ constant of the weight $w$, i.e.
\[[w]_{A_1}:=\esssup_{x\in X}\frac{Mw(x)}{w(x)}.\]

As is well-known, if $w\in A_p$ for some $1\leq p<\infty$ then $w=0$ a.e. or $w>0$ a.e. so that the interesting examples of such weights enjoy the latter property. Hence, whenever we have an $A_p$ weight we may assume that it is strictly positive. It is also easy to check from the definitions that $[w]_{A_1}\geq [w]_{A_p}\geq 1$ for every $1\leq p<\infty$.

For $1\leq p \leq \infty$ we denote by $p'$ the dual exponent of $p$, i.e. $1/p+1/p'=1$. In this definition $1/\infty$ means zero. A weight $w$ is said to belong to $A_{p,q}$ class for $1<p\leq q<\infty$ if it satisfies the condition
\[[w]_{A_{p,q}}:=\sup_{B}\left(\frac{1}{\mu(B)}\int_{B}w^q d\mu\right)\left(\frac{1}{\mu(B)}\int_{B}w^{-p'}d\mu\right)^{q/p'}<\infty.\]
The quantity $[w]_{A_{p,q}}\geq 1$ is called the $A_{p,q}$ constant of the weight $w$. 
A weight $w$ is said to belong to the $A_{1,q}$ class for $1\leq q<\infty$ if
\[Mw^q  \leq Cw^q \quad \text{a.e}, \]
and $[w]_{A_{1,q}}$ will again be the smallest constant $C$ that satisfies the above inequality. 

The following lemma is easy to check, and we leave the proof to the reader.
\begin{lemma}\label{lem:weights}
Let $1<p\leq q<\infty$, and denote $r=1+q/p'$ and $s=1+p'/q$. Then 
\begin{enumerate}
\item $[w]_{A_{p,q}}=[w^q]_{A_r}$. In particular, $w\in A_{p,q}$ if and only if $w^q\in A_r$;
\item $[w]_{A_{p,q}}=[w^{-1}]_{A_{q',p'}}^{q/p'}$. In particular, $w\in A_{p,q}$ if and only if $w^{-1}\in A_{q',p'}$;
\item $[w^{-p'}]_{A_s}=[w]_{A_{p,q}}^{p'/q}$. In particular, $w\in A_{p,q}$ if and only if $w^{-p'}\in A_{s}$;
\item $[w]_{A_{1,q}}=[w^q]_{A_1}$.  In particular, $w\in A_{1,q}$ if and only if $w^q\in A_1$.
\end{enumerate} 
\end{lemma}

\section{Fractional integral operators and the main result}\label{sec:fractionaloperator}
Let $(X,\rho,\mu)$ be a space of homogeneous type. We consider a fractional integral operators of order $0<\gamma <1$, defined by
\begin{equation}\label{def:fractional}
T_\gamma f(x)=\int_{X}K_\gamma(x,y)f(y) d\mu(y)
\end{equation}
where the kernel $K_\gamma$ is the positive function 
\begin{equation}\label{def:kernel}
K_\gamma(x,y)=\begin{cases} 
\mu \big(B(x,\rho(x,y))\big)^{\gamma-1},&  \text{when } x\neq y\\
\mu(\{x \})^{\gamma -1}, & \text{when } x=y .
\end{cases} 
\end{equation}

\begin{remarks}
\textit{1.} Suppose that $x\in X$ and $\mu(\{x \})=0$. Then our definition for $K_\gamma$ formally gives $K_\gamma(x,x)=+\infty$. However, we may write
\[T_\gamma f(x)=\int_{X\setminus \{ x \}}K_\gamma(x,y)f(y) d\mu(y) +
K_\gamma(x,x)f(x)\mu(\{ x\}),\]
and the latter term vanishes in case $\mu(\{x \})=0$ (by the usual interpretation $0\cdot \infty =0$). In fact, we may give an equivalent definition 
\[T_\gamma f(x):=\int_{X\setminus \{ x \}}\frac{f(y)d\mu(y)}{\mu \big(B(x,\rho(x,y))\big)^{1-\gamma}} +
f(x)\mu(\{x \})^{\gamma}.\]
If $\mu$ does not have atoms, 
the domain of integration $X\setminus \{ x\}$ may be replaced by $X$, and the extra term $f(x)\mu(\{x \})^{\gamma}$ does not appear.

\textit{2.} Consider $X=\R^n$ with the usual $n$-dimensional Lebesgue measure $d\mu =dx$. Then $\mu(B(x,\rho(x,y)))=\abs{B(x,\abs{x-y})}=C_n\abs{x-y}^{n}$ with a positive dimensional constant $C_n$. By the notation $\alpha:= n\gamma \in (0,n)$, our definition for $T_\gamma$ yields (up to a dimensional constant) the operator
\[I_\alpha f(x) = \int_{\R^n}\frac{f(y)}{\abs{x-y}^{n-\alpha}} dy,\]
which is the classical fractional integration (or the Riesz potential) of order $\alpha$ on $\R^n$. 
\end{remarks}

The following lemma shows that the operator $T_\gamma$ is an example of more general potential type operators studied in \cite{K11}. The proof of the Lemma only involves elementary estimations by the triangle inequality and Lemma~\ref{prop;balls;measures}, and we leave the details to the reader.

\begin{lemma}
 The operator $T_\gamma$ is an operator of potential type, i.e. the kernel $K_\gamma$, defined in \eqref{def:kernel},
satisfies the following monotonicity conditions: For every $k_2>1$ there exists $k_1>1$ such that 
\begin{equation}\label{kernel}
\begin{split}
&K_\gamma(x,y)\leq k_1 K_\gamma(x',y)\quad \text{whenever $\rho (x',y)\leq k_2\rho (x,y)$},\\
&K_\gamma(x,y)\leq k_1 K_\gamma(x,y')\quad \text{whenever $\rho (x,y')\leq k_2\rho (x,y)$}.
\end{split}
\end{equation} 
Moreover, there exists a geometric constant $C>0$ such that for all $x,y\in X$, $x\neq y$, 
\begin{equation}\label{eq:adjoint}
\frac{1}{C}K_\gamma(x,y)\leq K_\gamma(y,x) \leq CK_\gamma(x,y).
\end{equation}
\end{lemma}

We investigate the dependence of the operator norm of $T_\gamma$ on the $A_{p,q}$ constant of the weight in weighted spaces. Sharp weighted inequalities for the Riesz potentials $I_\alpha$ in the Euclidean spaces, acting on weighted Lebesgue spaces were obtained recently in \cite[Theorem 2.6]{LMPT10}. We use the ideas introduced there to extend this result into general spaces of homogeneous type. Our main result is the following.

\begin{theorem}\label{thm:main}
Suppose $(X,\rho,\mu)$ is a space of homogeneous type. Let $0<\gamma <1$ and suppose $1<p\leq q<\infty$ satisfy $1/p-1/q=\gamma$. Then
\[\| T_\gamma\|_{L^p(w^p)\to L^{q}(w^q)}\lesssim [w]_{A_{p,q}}^{(1-\gamma)\max\left\{ 1,\frac{p'}{q}\right\}}. \]
This estimate is sharp in any space $X$ with infinitely many points in the sense described in Section~\ref{sec:example}.
\end{theorem}
We mention that a similar qualitative result in a slightly less general setting can be found, for instance, in \cite[Theorem A on p. 412]{EKM} (see also the references mentioned there).

\section{A preliminary result}
Let us begin our investigations by motivating the restriction $1/p-1/q=\gamma$ imposed on exponents in Theorem~\ref{thm:main}.

Recall the well-known Hardy--Littlewood--Sobolev Theorem in the Euclidean space that if the Riesz potential $I_\alpha$ maps $L^p(\R^n)$ to $L^q(\R^n)$ for some $p$ and $q$, then we must have that the exponents are related by $1/p-1/q=\alpha/n$, and this condition is also sufficient to have a bounded operator. 
We record an analogous result for $T_\gamma$ in the present context. In fact, the following non-weighted result describes a necessary and sufficient condition for the exponents $p$ and $q$ for which $T_\gamma$ is a bounded operator from $L^p(X)$ to $L^q(X)$. This easy observation can probably be found elsewhere; cf. \cite[Theorem 6.2.2]{EKM} where the set-up is slightly less general, but in the lack of a suitable reference, we shall also provide a proof.

\begin{proposition}\label{lem:exponents}
Let $(X,\rho,\mu)$ be a space of homogeneous type. Let $0<\gamma<1$ and  $1< p,q<\infty$, and suppose $T_\gamma\colon L^p(X)\to L^q(X)$ is bounded. Then 
\[\mu(B)^{1/q-1/p+\gamma}\leq C<\infty \]
for all balls $B$. Moreover, the following is true:
\begin{enumerate}
\item[\textit{i)}] If $X$ belongs to the category 1 of Lemma~\ref{lem:three_choices}, then $T_\gamma\colon L^p(X)\to L^q(X)$ is bounded if and only if $1/p-1/q\leq \gamma$.
\item[\textit{ii)}] If $X$ belongs to the category 2 of Lemma~\ref{lem:three_choices},  then $T_\gamma\colon L^p(X)\to L^q(X)$ is bounded if and only if $1/p-1/q\geq \gamma$.
\item[\textit{iii)}] If $X$ belongs to the category 3 of Lemma~\ref{lem:three_choices}, then $T_\gamma\colon L^p(X)\to L^q(X)$ is bounded if and only if $1/p-1/q = \gamma$.
\end{enumerate}
\end{proposition}

\begin{remark}
If none of the cases $i-iii$ holds, then $X$ has finitely many points, and the boundedness of $T_\gamma$ is trivial.
\end{remark}

\begin{proof}
First assume that $T_\gamma\colon L^p(X)\to L^q(X)$ is bounded. Fix a ball $B=B(x_0,r)$ and suppose $x,y\in B$. Then 
\[\mu\big(B(x,\rho(x,y))\big)\leq \mu\big(B(x_0,3A_0^2r)\big)\lesssim \mu(B(x_0,r)).\]
Thus, for $y\neq x$,
\[K_\gamma(x,y)=\frac{1}{\mu(B(x,\rho(x,y)))^{1-\gamma}}\gtrsim 
\frac{1}{\mu(B)^{1-\gamma}}. \]
For $y=x$,
\[K_\gamma(x,y)=\frac{1}{\mu(\{x\})^{1-\gamma}}\geq \frac{1}{\mu(B)^{1-\gamma}}. \]
Thus, 
\[T_\gamma \chi_B(x)= \int_{B}K_\gamma(x,y)\,d\mu \gtrsim \frac{\mu(B)}{\mu(B)^{1-\gamma}}=\mu(B)^\gamma .\] 
It follows that
\[\|T_\gamma\|\mu(B)^{1/p}=\|T_\gamma\|\|\chi_B\|_{L^p(X)}\geq \|T_\gamma\chi_B \|_{L^q(X)} \geq 
\|\chi_B \, T_\gamma\chi_B\|_{L^q(X)}\gtrsim \mu(B)^\gamma\cdot \mu(B)^{1/q}\]
so that
\[\mu(B)^{1/q-1/p+\gamma}\lesssim \|T_\gamma\|<\infty . \]
This shows the necessity of the conditions imposed on the exponents in \textit{i}--\textit{iii}. The sufficiency in \textit{iii} follows from Theorem~\ref{thm:main} by choosing $w\equiv 1$.

For the sufficiency in \textit{i}, let $1< p,q<\infty$ be exponents such that $1/p-1/q\leq \gamma$, and let $q_0\geq q$ is such that $1/p-1/q_0=\gamma$. First, \textit{iii} implies that  $T_\gamma\colon L^p(X)\to L^{q_0}(X)$ is bounded, and the claimed boundedness follows since $L^{q_0}\subseteq L^q$ for $q_0\geq q$ by Remark~\ref{rem:organized}.

For the sufficiency in \textit{ii}, let $1< p,q<\infty$ be exponents such that $1/p-1/q\geq \gamma$, and let $q_0\leq q$ be such that $1/p-1/q_0=\gamma$. The claimed boundedness follows again by \textit{iii} and Remark~\ref{rem:organized}.
\end{proof}

\section{First steps of the proof: reduction to the weak-type result}\label{sec:firststep}

The proof of our main result, Theorem~\ref{thm:main}, entails several reductions. We start by observing that in order to obtain sharp bounds for the strong-type estimates it is sufficient to show sharp bounds for the weak-type ones. This follows from the investigations of \cite{K11}; cf. \cite{SWZ96, VW98}, where a large class of potential type operators were studied. 

Indeed, in \cite[Theorem 1.12]{K11}, it was shown that if $\sigma$ and $v$ are positive Borel-measures in a quasi-metric space $(X,\rho)$ which are finite on balls, and $T$ is a positive operator of the form 
\begin{equation}\label{def:operatorT}
T(f\, d\sigma)(x)=\int_{X}K(x,y)f(y)\, d\sigma(y), \quad x\in X,
\end{equation}
where the kernel $K$ is a non-negative function which satisfies the monotonicity conditions \eqref{kernel}, and $1<p\leq q<\infty$, then the boundedness
\[T(\cdot \,d\sigma)\colon L^p_\sigma \to L^q_v \]
is characterized by Sawyer-type testing conditions, which we recall below. 

In the present paper, we investigate the particular case $T=T_\gamma$, $d\sigma=w^{-p/(p-1)}d\mu$ and $dv=w^qd\mu$ where $w$ is an $A_{p,q}$-weight and $(X,\rho ,\mu)$ is a space of homogeneous type. In this particular case, Theorem 1.12 of  \cite{K11} says that 
\[T_\gamma\colon L^p(w^p) \to L^q (w^q)\]
is a bounded operator if and only if the functions $\sigma:=w^{-p/(p-1)}$ and $v:=w^q$ satisfy the (local) testing conditions that
\[[\sigma, v]_{S_{p,q}}:= \sup_{Q}\sigma(Q)^{-1/p}\| \chi_{Q}T_\gamma(\chi_Q\sigma)\|_{L^q(v)} <\infty \]
and
\[ [v,\sigma]_{S_{q',p'}}:= \sup_{Q}v(Q)^{-1/q'}\| \chi_{Q}T_\gamma(\chi_Qv)\|_{L^{p'}(\sigma)} <\infty , \]
where the supremum is over all so-called dyadic cubes $Q\in\bigcup_{t=1}^{K}\mathscr{D}^t$ (for precise definitions, see \cite[Section 2.2]{K11}). 
The proof further shows that
\begin{equation}\label{eq:characterization(1)}
\| T_\gamma\|_{L^p(w^p)\to L^q(w^q)}\approx [\sigma,v]_{S_{p,q}}+[v,\sigma]_{S_{q',p'}} . 
\end{equation}
Moreover, by the characterization of the weak-type two-weight estimate \cite[Theorem 5.2]{K11},
\begin{equation}\label{eq:characterization(2)}
\| T_\gamma\|_{L^p(w^p)\to L^{q,\infty}(w^q)}\approx [v ,\sigma]_{S_{q',p'}}. 
\end{equation}

Using these characterizations and the fact that $T_\gamma$ is self-adjoint, we get the following.

\begin{proposition}\label{lem:reduction}
\[\| T_\gamma\|_{L^p(w^p)\to L^{q}(w^q)}\approx \| T_\gamma\|_{L^p(w^p)\to L^{q,\infty}(w^q)}
+\| T_\gamma\|_{L^{q'}(w^{-q'})\to L^{p',\infty}(w^{-p'})}\]
\end{proposition}
\begin{proof}
Denote $u:=w^p$ so that $\sigma=u^{1-p'}$ which is equivalent to $u=\sigma^{1-p}$. With this notation and $v:=w^q$, \eqref{eq:characterization(2)} becomes
\[\| T_\gamma\|_{L^p(u)\to L^{q,\infty}(v)}\approx [v ,u^{1-p'}]_{S_{q',p'}}.\]
Thus,
\[[\sigma,v]_{S_{p,q}}=[u^{1-p'},v^{(1-q')(1-q)}]_{S_{p,q}}
\approx \| T_\gamma\|_{L^{q'}(v^{1-q'})\to L^{p',\infty}(u^{1-p'})}.\]
Then combine this and \eqref{eq:characterization(2)} with \eqref{eq:characterization(1)} to make the final conclusion.
\end{proof}

By Proposition~\ref{lem:reduction}, the proof of Theorem~\ref{thm:main} is completed by the following proposition.

\begin{proposition}[Weak-type estimate]\label{prop:main}
Let $0<\gamma <1$ and suppose $1\leq p\leq q<\infty$ satisfy $1/p-1/q=\gamma$. Then
\[\| T_\gamma\|_{L^p(w^p)\to L^{q,\infty}(w^q)}\lesssim [w]_{A_{p,q}}^{1-\gamma}. \]
This estimate is sharp in any space $X$ in the sense described in Section~\ref{sec:example}.
\end{proposition}

\begin{proof}[Proof of Theorem~\ref{thm:main} assuming Proposition~\ref{prop:main}] Note that if $1/p-1/q=\gamma$, then also $1/q'-1/p'=\gamma$. By Proposition~\ref{lem:reduction}, we have that
\begin{align*}
\| T_\gamma\|_{L^p(w^p)\to L^{q}(w^q)}& \lesssim
\| T_\gamma\|_{L^p(w^p)\to L^{q,\infty}(w^q)}
+\| T_\gamma\|_{L^{q'}(w^{-q'})\to L^{p',\infty}(w^{-p'})}\\
&\lesssim [w]_{A_{p,q}}^{1-\gamma}+[w^{-1}]_{A_{q',p'}}^{1-\gamma}
= [w]_{A_{p,q}}^{1-\gamma}+[w]_{A_{p,q}}^{(1-\gamma)p'/q}\\
&\leq 2[w]_{A_{p,q}}^{(1-\gamma)\max\left\{1,\frac{p'}{q}\right\}},
\end{align*}
where we used Lemma~\ref{lem:weights}(ii).
\end{proof}

\section{Proof of the weak-type result via extrapolation}\label{sec:weaktyperesult}

To prove Proposition~\ref{prop:main}, we will perform yet another reduction where we use the following sharp weak-type version of an extrapolation theorem for $A_{p,q}$ weights.

\begin{theorem}\label{thm:extrapolation(2)}
Let $T$ be an operator defined on an appropriate class of functions (e.g. bounded functions with bounded support). Suppose that for some pair $(p_0,q_0)$ of exponents $1\leq p_0\leq q_0<\infty$, $T$ satisfies the weak-type inequality
\begin{equation}\label{eq:weaktype}
\| Tf\|_{L^{q_0, \infty}(w^{q_0})}\leq C[w]_{A_{p_0,q_0}}^\alpha \| f\|_{L^{p_0}(w^{p_0})} 
\end{equation}
for all weights $w\in A_{p_0,q_0}$ and with some $\alpha>0$. Then,
\[\| Tf\|_{L^{q,\infty}(w^{q})}\leq C[w]_{A_{p,q}}^{\alpha\max \{ 1,\frac{q_0}{p_0'}\frac{p'}{q}\}} \| f\|_{L^{p}(w^{p})} \]
for all weights $w\in A_{p,q}$ and all pairs $(p,q)$ of exponents that satisfy
\[\frac{1}{p}-\frac{1}{q}=\frac{1}{p_0}-\frac{1}{q_0}.\]
\end{theorem}

The Euclidean version of this extrapolation theorem can be found in \cite[Corollary 2.2]{LMPT10} where it is shown to follow from the corresponding sharp strong-type extrapolation result. The Euclidean proofs can be adapted into the present context. We will comment on this in Section~\ref{sec:extrapolation}. 

As for now, assume Theorem~\ref{thm:extrapolation(2)}. We observe that in order to successfully apply the Theorem and obtain the desired exponent $1-\gamma$ for all pairs $(p,q)$ of exponents in the norm estimate of Proposition~\ref{prop:main}, it becomes necessary that we show the weak-type inequality \eqref{eq:weaktype} for $T_\gamma$ with exponents $p_0=1$ and $q_0=1/(1-\gamma)$; for any other pair, the extrapolation theorem would only give the positive result for a limited range of exponents, i.e. for the ones with $q_0p'/p_0'q\leq 1$, and for large $p$, the exponent obtained by the extrapolation would be strictly larger. Thus, we state the following lemma.

\begin{mainlemma}\label{lem:weak1}
Let $0\leq u\in L^1_{\loc}(X,\mu)$ be a weight. A fractional integral operator $T_\gamma, 0<\gamma<1$, satisfies the weak-type estimate
\[\| T_\gamma f\|_{L^{q_0,\infty}(u)}\leq C\|f\|_{L^{1}((Mu)^{1/q_0})} \]
with $q_0=1/(1-\gamma)$. As a consequence, 
\[\| T_\gamma \|_{L^{q_0, \infty}(w^{q_0})}\leq C[w]_{A_{1,q_0}}^{1-\gamma} \| f\|_{L^{1}(w)} \]
for all weights $w\in A_{1,q_0}$.
\end{mainlemma}

The Main Lemma together with the extrapolation result gives Proposition~\ref{prop:main} which in turn leads to strong-type estimates and complete the proof of our main result, Theorem~\ref{thm:main}, as already described:

\begin{proof}[Proof of Proposition~\ref{prop:main} assuming Main lemma~\ref{lem:weak1}]
We apply Theorem~\ref{thm:extrapolation(2)} with exponents $p_0=1$ and $q_0=1/(1-\gamma)$, and $\alpha=1-\gamma$. First note that 
\[\alpha \max\left\{ 1,\frac{q_0}{p_0'}\frac{p'}{q} \right\}=1-\gamma. \]
Theorem~\ref{thm:extrapolation(2)} together with Main Lemma~\ref{lem:weak1} show that
\[\| T_\gamma f\|_{L^{q, \infty}(w^{q})}\leq C[w]_{A_{p,q}}^{1-\gamma} \| f\|_{L^{p}(w^p)} \]
for all weights $w\in A_{p,q}$ and all exponents $1<p\leq q<\infty$ that satisfy
\[\frac{1}{p}-\frac{1}{q}=\frac{1}{p_0}-\frac{1}{q_0}=1-(1-\gamma)=\gamma . \]
\end{proof}

We are left to prove Main Lemma~\ref{lem:weak1}. The proof follows the corresponding Euclidean proof given in \cite{LMPT10} except that we need to put out some extra work with the technical details when working with general doubling measures. 
 
\begin{proof}[Proof of Main lemma~\ref{lem:weak1}]
We recall that $\|\cdot \|_{L^{p,\infty}(u)}$ is equivalent to a norm when $p>1$. Hence, we may use the triangle inequality as follows
\begin{align}\label{eq:normest}
\| T_\gamma f\|_{L^{q_0,\infty}(u)} &
\leq C_{q_0}\int_{X}\abs{f(y)}\| K_\gamma(\cdot ,y)\|_{L^{q_0,\infty}(u)} d\mu (y).
\end{align}
Fix $y\in X$. First note that for all $x\in X$ 
\[K_\gamma(x,y)\leq \frac{1}{\mu(\{y \})^{1-\gamma}}.\]
We then calculate
\begin{align*}
\| K_\gamma(\cdot ,y)\|_{L^{q_0,\infty}(u)} & = 
\sup_{\lambda >0} \lambda 
\left[u\left(\left\{ x\in X\colon K_\gamma(x,y)>\lambda \right\}\right)\right]^{1/q_0}\\
&= \sup_{0<\lambda <\mu(\{ y\})^{\gamma -1}} \lambda 
\left[u\left(\left\{ x\in X\colon K_\gamma(x,y)>\lambda \right\}\right)\right]^{1/q_0}\\
& = \left[ \sup_{0<\lambda <\mu(\{ y\})^{\gamma -1}} \lambda^{q_0} u\left(\left\{ x\in X\colon K_\gamma(x,y)^{\frac{1}{\gamma -1}}<\lambda^{\frac{1}{\gamma -1}} \right\} \right)\right]^{1/q_0}\\
& =\left[ \sup_{t>\mu(\{ y\})} \frac{1}{t} u\left(\left\{ x\in X, x\neq y\colon \mu\big(B(x,\rho(x,y))\big)<t \right\} \right)\right]^{1/q_0}\!\!\text{since $q_0=1/(1-\gamma)$}\\
& \leq \left[ \sup_{t>\mu(\{ y\})} \frac{1}{t} u\left(\left\{ x\in X\colon \mu\big(B(y,\rho(x,y))\big)<Ct \right\} \right)\right]^{1/q_0}\\
& \leq  C^{1/q_0}\left[ \sup_{t>\mu(\{ y\})} \frac{1}{t} u\left(\left\{ x\in X\colon \mu\big(B(y,\rho(x,y))\big)<t \right\} \right)\right]^{1/q_0}, \quad C=C(A_0,\mu).
\end{align*}
The second to last estimate is true since $\mu\big(B(x,\rho(x,y))\big)<t$ implies that $\mu\big(B(y,\rho(x,y))\big)\leq \mu\big(B(x,2A_0\rho(x,y))\big)<Ct$, $C=C(A_0,\mu)$, by the doubling property. 

For a fixed $y\in X$, denote $E_t:=\left\{ x\in X\colon \mu\big(B(y,\rho(y,x))\big)<t \right\}$. Note that $y\in E_t$ for all $t>0$. We make the following technical observation.

\begin{lemma}
Given $y\in X$ and $t>0$, consider the set $E_t:=\left\{ x\in X\colon \mu\big(B(y,\rho(y,x))\big)<t \right\}$ and the quantity
\[r_y(t):=\sup\{r\geq 0\colon \mu\big(B(y,r)\big)<t\}\in [0,\infty].\]
Here it is understood that $B(y,0)=\emptyset$ so that the supremum always exists. Then the following is true:

1) If $x_1\in E_t$ for some $x_1\neq y$, then $x\in E_t$ for all $x$ with $\rho(y,x)\leq \rho(y,x_1)$, and $r_y(t)>0$. 

2) If $x_2\notin E_t$ for some $x_2$, then $x\notin E_t$ for all $x$ with $\rho(y,x)\geq \rho(y,x_2)$, and $r_y(t)<\infty$. 

3) If $r_y(t)=0$, then $E_t=\{y\}$. 

4) If $r_y(t)=\infty$, then $E_t=X$ and $\mu(X)\leq t$. 

5) If $0<r_y(t)<\infty$, then the set $E_t$ is one of the two choices $B(y,r_y(t))$ and $\bar{B}(y,r_y(t))$. Moreover, $\mu(E_t)\leq t$.
\end{lemma}

\textit{Case 1:} Consider such $t>\mu(\{ y\})$ for which there exists $x\in X, x\neq y$, with $x\in E_t$. 
Then $r_y(t)>0$, and $E_t$ is one of the three choices, $B(y,r_y(t))$ or $\bar{B}(y,r_y(t))$ or $X$, and $\mu(E_t)\leq t$. Recall that in case $E_t=X$, the condition $\mu(E_t)\leq t<\infty$ implies that $X=B(y,R)$ for some $0<R<\infty$. Hence, for such $t$ we have
\begin{equation*}\label{eq:easycase}
 \frac{1}{t} u\left(E_t \right)\leq \frac{u(E_t)}{\mu(E_t)}\leq Mu(y) .
\end{equation*}

\textit{Case 2:} Then consider $t>\mu(\{ y\})$ with $E_t=\{ y\}$. If $\mu(\{ y\})=0$, then $u(E_t)/t=0$. Recall that $\mu(\{ y\})>0$ implies that $\{ y\}=B(y,\varepsilon)$ for some $\varepsilon>0$.  For such $t$ we have
\[ \frac{1}{t} u\left(E_t \right)\leq \frac{u(B(y,\varepsilon))}{\mu(B(y,\varepsilon))}\leq Mu(y).\]

Altogether we have obtained that the inner norm in \eqref{eq:normest} satisfies
\begin{align*}
\| K_\gamma(\cdot ,y)\|_{L^{q_0,\infty}(u)} & \leq 
C_{q_0}\left(\sup_{t>\mu(\{y\})} \frac{1}{t} u\left(E_t \right)\right)^{1/q_0}\leq 
C_{q_0}\left(Mu(y)\right)^{1/q_0},
\end{align*}
and consequently,
\begin{align*}
\| T_\gamma f\|_{L^{q_0,\infty}(u)} & \leq 
C_{q_0}\int_{X}\abs{f(y)}\left(Mu(y)\right)^{1/q_0} d\mu (y) = C_{q_0}\|f\|_{L^1((Mu)^{1/q_0})}.
\end{align*}
This is the first assertion.

Then suppose that $w\in A_{1,q_0}$ and denote $u:=w^{q_0}$. Recall the $A_{1,q_0}$ condition for $w$,
\[Mu\leq [w]_{A_{1,q_0}}u\quad \text{a.e}, \]
and that $q_0=1/(1-\gamma)$. From this and the first assertion we may deduce
\begin{align*}
\| T_\gamma f\|_{L^{q_0,\infty}(w^{q_0})} & \leq 
C_{q_0}\int_{X}\abs{f}\left(Mu\right)^{1/q_0} d\mu \leq 
C_{q_0}[w]_{A_{1,q_0}}^{1/q_0}\int_{X}\abs{f}w d\mu\\
& = C_{q_0}[w]_{A_{1,q_0}}^{1-\gamma}\|f\|_{L^1(w)},
\end{align*}
which completes the proof.
\end{proof}

\section{Extrapolation}\label{sec:extrapolation}

In this section we justify the use of the sharp extrapolation theorem \ref{thm:extrapolation(2)} by verifying that the Euclidean proof of the theorem can be adapted into our situation.

To this end, we recall that in \cite[Corollary 2.2]{LMPT10} and the Euclidean setting, the sharp weak-type extrapolation theorem \ref{thm:extrapolation(2)}, was deduced from the corresponding sharp strong-type extrapolation result, which we recall below. To show this deduction, the authors used an idea from Grafakos and Martell \cite[Theorem 6.1]{GM04} which is very general and applies to our situation. Thus, we may complete the proof of Theorem~\ref{thm:extrapolation(2)} by the following theorem.

\begin{theorem}
\label{thm:extrapolation}
Let $T$ be an operator defined on an appropriate class of functions (e.g. bounded functions with bounded support). Suppose that for some exponents $1\leq p_0\leq q_0<\infty$, $T$ satisfies
\[\| Tf\|_{L^{q_0}(w^{q_0})}\leq C[w]_{A_{p_0,q_0}}^\alpha \| f\|_{L^{p_0}(w^{p_0})} \]
for all weights $w\in A_{p_0,q_0}$ and some $\alpha>0$. Then,
\[\| Tf\|_{L^{q}(w^{q})}\leq C[w]_{A_{p,q}}^{\alpha\max \{ 1,\frac{q_0}{p_0'}\frac{p'}{q}\}} \| f\|_{L^{p}(w^{p})} \]
holds for all weights $w\in A_{p,q}$ and all exponents $1<p\leq q<\infty$ that satisfy
\[\frac{1}{p}-\frac{1}{q}=\frac{1}{p_0}-\frac{1}{q_0}.\]
\end{theorem}

The original qualitative version of this extrapolation result in the Euclidean space is due to Harboure, Mac\'ias and Segovia \cite{HMS88}. The sharp version in the Euclidean space can be found in \cite[Theorem 2.1]{LMPT10}. To show the metric space version, Theorem~\ref{thm:extrapolation}, we may follow, from line to line, the Euclidean proof from \cite{LMPT10} except that we need the following result for $A_p$ weights which is the main tool in the proof. 

\begin{lemma}\label{lem:lemma3}
Suppose $\mu$ is a doubling measure on $X$. Let $1\leq r_0<r <\infty$ and $w\in A_r$. Then for any $g\geq 0$, $g\in L^{(r/r_0)'}(w)$, there exists a function $G\in L^{(r/r_0)'}(w)$ with the properties that
\begin{enumerate}
\item $G\geq g$;
\item $\|G \|_{L^{(r/r_0)'}(w)}\leq 2\|g \|_{L^{(r/r_0)'}(w)}$;
\item $Gw\in A_{r_0}$; moreover, $[Gw]_{A_{r_0}}\leq C[w]_{A_r}$ where $C>0$ depends only on $X,\mu ,r_0$ and $r$.
\end{enumerate}
\end{lemma}

A qualitative version of Lemma~\ref{lem:lemma3} in the Euclidean space first appeared in \cite{GC83}. A quantitative version; cf. the results in \cite{DGPP2006}, uses a suitable sharp version of the celebrated Rubio de Francia algorithm, a very general technique, and Buckley's theorem \cite{Buckley:93} on the sharp dependence of $\|M\|_{L^p(w)}$ on $[w]_{A_p}$ in Muckenhoupt's theorem for the Hardy--Littlewood maximal operator. Buckley's result in a space of homogeneous type was shown in \cite[Proposition 7.13]{oma}. After this, the proof of Lemma~\ref{lem:lemma3} follows, again from line to line, the Euclidean proof, and we may refer the reader to the original proof given in \cite{DGPP2006}. 

\section{Sharpness of the result}\label{sec:example}
In this final section, we show that the exponent $1-\gamma$ in the estimate
\begin{equation}\label{eq:estimate1}
\|T_\gamma f\|_{L^{q,\infty}(w^q)}\lesssim [w]^{1-\gamma}_{A_{p,q}} \|f\|_{L^p(w^p)}
\end{equation}
from Proposition~\ref{prop:main} is best possible in the sense described as follows. This also implies that the exponent $(1-\gamma)\max \{1,p'/q\}$ in the norm estimate in Theorem~\ref{thm:main} is sharp. In fact, we will show the following.

\begin{proposition}\label{prop:example}
Let $(X,\rho,\mu)$ be a space of homogeneous type with the property that $\# X=\infty$. Then, there exists a family $\{w_t\colon 0<t<1\}$ of weights such that 
\[[w_t]_{A_{p,q}}\approx \frac{1}{t},\]
and
\begin{equation}\label{est:UP}
\|T_\gamma\|_{L^{p}(w_t^p)\to L^{q,\infty}(w_t^q)}\gtrsim [w_t]_{A_{p,q}}^{1-\gamma}.
\end{equation}
Consequently, if $\|T_\gamma\|_{L^{p}(w^p)\to L^{q,\infty}(w^q)}\leq \phi([w]_{A_{p,q}})$ for some increasing $\phi\colon [1,\infty)\to (0,\infty)$, then $\phi(s)\gtrsim s^{1-\gamma}$. In particular, for any $\varepsilon>0$, we have that
\begin{equation*}
\sup_{w\in A_{p,q}}
\frac{\|T_\gamma\|_{L^{p}(w^p)\to L^{q,\infty}(w^q)}}{[w]_{A_{p,q}}^{(1-\gamma)-\varepsilon}}=\infty .
\end{equation*}
\end{proposition}

We will first observe that Proposition~\ref{prop:example} follows from the following lemma.

\begin{lemma}[Reduction]\label{lem:reduction2}
Let $(X,\rho,\mu)$ be a space of homogeneous type with the property that $\# X=\infty$. Then, for every $0<t<1$, there exists a weight $u_t$ and a function $f_t\neq 0$ such that $[u_t]_{A_1}\approx 1/t$ and 
\begin{equation}\label{eq:estimateUP}
\|T_\gamma (f_tu_t^{\gamma})\|_{L^{q,\infty}(u_t)}\gtrsim [u_t]_{A_1}^{1-\gamma}\|f_t\|_{L^p(u_t)}.
\end{equation}
\end{lemma}
\begin{proof}[Proof of Proposition~\ref{prop:example} assuming Lemma~\ref{lem:reduction2}]
First, note that 
\[\|f_t\|_{L^p(u_t)} = \|f_tu_t^{\gamma}\|_{L^p(u_t^{p/q})}\] since $1/p-1/q=\gamma$. By replacing $f_t$ with $f_tu_t^{-\gamma}$, \eqref{eq:estimateUP} becomes 
\begin{equation}\label{eq:estimateUP2}
\|T_\gamma f_t\|_{L^{q,\infty}(u_t)}\gtrsim [u_t]_{A_{1}}^{1-\gamma}\|f_t\|_{L^p(u_t^{p/q})}.
\end{equation}
We denote $w_t:=u_t^{1/q}$ and observe that, by Lemma~\ref{lem:weights}(i), $[u_t]_{A_1}=[w_t^q]_{A_1}\geq [w_t^q]_{A_{1+q/p'}}=[w_t]_{A_{p,q}}$. Thus, \eqref{eq:estimateUP2} yields
\[\|T_\gamma f_t\|_{L^{q,\infty}(w_t^q)}\gtrsim 
[w_t^q]^{1-\gamma}_{A_{1}}\|f_t\|_{L^p(w_t^p)}\geq 
[w_t]^{1-\gamma}_{A_{p,q}}\|f_t\|_{L^p(w_t^p)}. \]
This shows the estimate \eqref{est:UP}. Moreover, we have that
\[[w_t]^{1-\gamma}_{A_{p,q}}\|f_t\|_{L^p(w_t^p)}\gtrsim
\|T_\gamma f_t\|_{L^{q,\infty}(w_t^q)}\gtrsim  
[w_t^q]^{1-\gamma}_{A_{1}}\|f_t\|_{L^p(w_t^p)}\geq 
[w_t]^{1-\gamma}_{A_{p,q}}\|f_t\|_{L^p(w_t^p)}\]
so that
\[[w_t]_{A_{p,q}}\approx [w_t^q]_{A_1}=[u_t]_{A_1}\approx\frac{1}{t}.\]
Finally, let $\phi\colon [1,\infty)\to (0,\infty)$ be an increasing function such that $\|T_\gamma\|_{L^{p}(w^p)\to L^{q,\infty}(w^q)}\leq \phi([w]_{A_{p,q}})$. Then, in particular, for every $0<t<1$ and a large $C$,
\[\phi(C/t)\geq \phi([w_t]_{A_{p,q}})\geq \|T_\gamma\|_{L^{p}(w^p)\to L^{q,\infty}(w^q)}\gtrsim [w_t]^{1-\gamma}_{A_{p,q}}\gtrsim t^{\gamma-1}\]
so that for every $s:=C/t\in(C,\infty)$,
\[\phi(s)\gtrsim (C/s)^{\gamma-1}\gtrsim s^{1-\gamma}.\]
\end{proof}

We are left to prove Lemma \ref{lem:reduction2}. The proof consists of several steps. 
We start with the following definitions.

\begin{definition}[$\varepsilon$-point]
We say that a  point $x\in X$ is \textit{an $\varepsilon$-point} for $\varepsilon>0$, if there exists $R>0$ such that
\begin{equation}\label{def:eps_point}
\mu(B(x,R)) >\varepsilon^{-1}\mu(\{x\}).
\end{equation}
The key observation in our investigations in this section is that the property $\# X=\infty$ implies, by Lemma~\ref{lem:property}, the existence of an $\varepsilon$-point for every $\varepsilon>0$. 
\end{definition}

\begin{definition}[Power weights]\label{def:powerweight}
For $0<t<1$, let $x_t\in X$ be an $\varepsilon=\varepsilon(t)$-point (for a small $\varepsilon(t)>0$ to be fixed). We define
\begin{equation}\label{eq:weights}
u_t(x):=\frac{1}{\mu(B(x_t,\rho(x,x_t)))^{1-t}},
\end{equation}
where it is agreed that $B(x,0)=\{x\}$ for all $x\in X$.
\end{definition}

The small positive number $\varepsilon(t)$ will vary in the different lemmata below, until we fix it at the end of the proof.

\begin{lemma}\label{lem:integraloverB0}
Let $0<t<1$ and suppose that $x_t$ is an $\varepsilon$-point with $\varepsilon =(2C_\mu)^{-3/t}$. Then, for any ball $B=B(x_t,R)$,
\[u_t(B):=\int_{B}u_t\,d\mu\lesssim \frac{\mu(B)^t}{t}.\]
Moreover, if $\mu(B(x_t,R))>\varepsilon^{-1}\mu(\{x_t\})$, then
\[u_t(B)\gtrsim \frac{\mu(B)^t}{t}.\]
\end{lemma}
\begin{proof}
Fix $B=B(x_t,R)$.

\textit{Case 1:} First assume that $\mu(B)\leq  2\mu(\{x_t\})$, so that only the first assertion requires a proof. Note that now $\mu(\{x_t\})>0$, and thus $\{x_t\}=B(x_t,\varepsilon)$ for some $\varepsilon>0$. We have
\begin{align*}
\int_{B}u_t\,d\mu & = \int_{\{x_t\}}u_t\,d\mu + \int_{B\setminus B(x_t,\varepsilon)}u_t\,d\mu \leq \mu(\{x_t\})^t + \frac{\mu(B)-\mu(\{x_t\})}{\mu(\{x_t\})^{1-t}}\leq 2\mu(\{x_t\})^t \lesssim \frac{\mu(B)^t}{t}.
\end{align*}
\textit{Case 2:} We may then assume that $\mu(B)> 2\mu(\{x_t\})$. We choose a decreasing sequence $(r_k)$ of radii as follows: Let $r_0=R$. Then let $k_1\geq 1$ be the smallest integer such that $\mu(B(x_t,2^{-k_1}R))<2^{-1}\mu(B(x_t,r_0))$, and set $r_1:=2^{-k_1}R$. Note that since $\mu(B(x_t,r_0))> 2\mu(\{x_t\})$, such an $r_1$ exists. 
Having chosen $k_m$ and $r_m$ in this fashion, let $k_{m+1}> k_m$ be the smallest integer such that $\mu(B(x_t,2^{-k_{m+1}}R))< 2^{-1}\mu(B(x_t,r_m))$, and set $r_{m+1}:= 2^{-k_{m+1}}R$. Note that if $\mu(\{x_t\})=0$, we may keep sub-dividing infinitely many times. Otherwise, we stop iterating at the step $K\geq 1$ for which $\mu(B(x_t,r_{K-1}))> 2\mu(\{x_t\})$ and $\mu(B(x_t,r_K))\leq 2\mu(\{x_t\})$. 

Note that we have
\begin{equation}\label{eq:anestimate1}
\mu(B(x_t,r_{i+1}))< 2^{-1} \mu(B(x_t,r_{i})).
\end{equation} 
On the other hand, since $r_{i+1}$ is, by choice, the largest number of the form $2^{-k_1}r$ with the property \eqref{eq:anestimate1}, $2r_{i+1}$ satisfies the inverse estimate, and thus
\begin{equation}\label{eq:anestimate2}
\mu(B(x_t,r_{i}))\leq 2 \mu(B(x_t,2r_{i+1}))\leq 2C_\mu \mu(B(x_t,r_{i+1})),
\end{equation} 
where the second estimate follows by the doubling property. 

For the first assertion, we consider two cases: First assume that $\mu(\{x_t\})=0$. We may then keep sub-divining infinitely many times, and
\begin{align*}
\int_{B}u_t\,d\mu & = \sum_{i= 0}^{\infty}\int_{B(x_t,r_{i})\setminus B(x_t,r_{i+1})}\frac{d\mu(y)}{\mu(B(x_t,\rho(y,x_t)))^{1-t}}
\leq  \sum_{i= 0}^{\infty} \frac{\mu(B(x_t,r_{i}))}{\mu(B(x_t,r_{i+1}))^{1-t}}\\
& \lesssim \sum_{i=0}^{\infty}\mu(B(x_t,r_{i+1}))^{t} \leq \mu(B(x_t,r_0))^{t}\big(1+2^{-t}+2^{-2t}+\ldots \big)\\
& = \frac{\mu(B(x_t,R))^{t}}{1-2^{-t}} \lesssim \frac{\mu(B)^{t}}{t} 
\end{align*}
where we used \eqref{eq:anestimate2} in the second estimate and \eqref{eq:anestimate1} in the second-to-last estimate. 

Then assume that $\mu(\{x_t\})>0$ and let $K$ denote the step when the iteration ends. Then
\begin{align*}
\int_{B}u_t\,d\mu & = \sum_{i=0}^{K-1}
\int_{B(x_t,r_{i})\setminus B(x_t,r_{i+1})}\frac{d\mu(y)}{\mu(B(x_t,\rho(y,x_t)))^{1-t}} + 
\int_{B(x_t,r_{K})}\frac{d\mu(y)}{\mu(B(x_t,\rho(y,x_t)))^{1-t}}\\
& =:I_1+I_2.
\end{align*}
The term $I_1$ is estimated as in the first case only that we now have a finite sum instead of an infinite one. Recall that the iteration stops when $\mu(B(x_t,r_K))\leq 2\mu(\{x_t\})$, so that the ball $B(x_t,r_K)$ in the term $I_2$ is in the regime of the Case 1. This completes the proof for the first assertion.
 
For the second assertion, suppose $\mu(B(x_t,R))>(2C_\mu)^{3/t}\mu(\{x_t\})$. Let $K$ be an integer such that
\[K-1<\frac{1}{t}\leq K.\] 
By iterating \eqref{eq:anestimate2}, we see that
\[\mu(B(x_t,r_K))\geq (2C_\mu)^{-K} \mu(B(x_t,r_0))\geq (2C_\mu)^{-1} (2C_\mu)^{-1/t}\mu(B(x,R))>2\mu(\{x_t\}),\]
which shows that the iteration proceeds at least $K$ times. Also note that by $K\geq 1/t$, we have
\begin{equation}\label{est:forK}
(2C_\mu)^{-Kt}\leq 2^{-1}. 
\end{equation}
Thus,
\begin{align*}
\int_{B}u_t\,d\mu & = \sum_{i=0}^{K-1} \int_{B(x_t,r_{i})\setminus B(x_t,r_{i+1})}\frac{d\mu(y)}{\mu(B(x_t,\rho(y,x_t)))^{1-t}} + \int_{B(x_t,r_{K})}\frac{d\mu(y)}{\mu(B(x_t,\rho(y,x_t)))^{1-t}}\\
& \geq \sum_{i=0}^{K-1}  \frac{\mu(B(x_t,r_{i}))-\mu(B(x_t,r_{i+1}))}{\mu(B(x_t,r_{i}))^{1-t}}.
\end{align*}
Note that here $\mu(B(x_t,r_{i}))-\mu(B(x_t,r_{i+1}))\geq 2^{-1}\mu(B(x_t,r_{i}))$ by \eqref{eq:anestimate1}. Thus, by \eqref{eq:anestimate2} and \eqref{est:forK},
\begin{align*}
\int_{B}u_t\,d\mu & \geq 
2^{-1}\sum_{i=0}^{K-1}\mu(B(x_t,r_{i}))^{t}
\geq 2^{-1}\mu(B(x_t,r_{0}))^{t}\Big(1+(2C_\mu)^{-t}+\ldots +(2C_\mu)^{-(K-1)t} \Big)\\
& = 2^{-1}\mu(B(x_t,R))^{t}\frac{1-(2C_\mu)^{-Kt}}{1-(2C_\mu)^{-t}} \gtrsim \frac{\mu(B)^{t}}{t}.
\end{align*}
\end{proof} 

\begin{lemma}\label{lem:A1}
Let $0<t<1$ and suppose that $x_t$ is an $\varepsilon$-point with $\varepsilon =(2C_\mu)^{-3/t}$. Then
\[[u_t]_{A_1}\approx \frac{1}{t}.\]
\end{lemma}
\begin{proof}
To show the estimate $\lesssim$, it suffices to show that for a.e. $x\in X$ and all balls $B\ni x$,
\[\frac{1}{\mu(B)}\int_{B}u_t\,d\mu\lesssim \frac{u_t(x)}{t}.\]
To this end, fix $x\in X$ and a ball $B=B(y,r)\ni x$. 

\textit{Case 1:} First assume that $r\leq (4A_0^2)^{-1}\rho(x,x_t)$. Note that if  $x_t$ is not a point mass, it suffices to consider points $x\neq x_t$; otherwise, for $x=x_t$, the restriction on $r$ (formally) reduces to considering only the ball $B(x_t,0)$ which is interpreted as the singleton $\{x_t\}$. Then, for $z\in B$ we have that $\rho(z,x_t)\geq (2A_0)^{-1}\rho(x,x_t)$, and thus
\begin{align*}
\frac{1}{\mu(B)}\int_{B}u_t\,d\mu &= 
\frac{1}{\mu(B)}\int_{B}\frac{d\mu(z)}{\mu(B(x_t,\rho(z,x_t)))^{1-t}}\leq 
\frac{1}{\mu(B(x_t,(2A_0)^{-1}\rho(x,x_t)))^{1-t}} \\
& = 
\left(\frac{\mu(B(x_t,\rho(x,x_t)))}{\mu(B(x_t,(2A_0)^{-1}\rho(x,x_t)))} \right)^{1-t}u_t(x)\lesssim u_t(x)\leq \frac{u_t(x)}{t}.
\end{align*}

\textit{Case 2:} Then assume that $r> (4A_0^2)^{-1}\rho(x,x_t)$. (This also includes the case $x=x_t$ if $x_t$ is a point mass, and in this case we consider any $r>0$.) Consider the balls $\hat{B}:=B(x_t,R), R:=6A_0^4 r$, and $\tilde{B}:=B(y,2A_0R)$. It is easy to see that $B\subseteq \hat{B}\subseteq \tilde{B}$, and the doubling property implies that
\begin{equation}\label{eq:measures}
\mu(B)\leq \mu(\hat{B})\leq \mu(\tilde{B})\leq C\mu(B), \quad C=C(A_0,\mu). 
\end{equation}
Thus, by Lemma~\ref{lem:integraloverB0}, we conclude with
\begin{align*}
\frac{1}{\mu(B)}\int_{B}u_t\,d\mu & \leq \frac{C}{\mu(\hat{B})}\int_{\hat{B}}u_t\,d\mu 
\lesssim \frac{1}{\mu(\hat{B})}\frac{\mu(\hat{B})^t}{t} =\frac{1}{t\mu(B(x_t,R))^{1-t}}
\leq \frac{u_t(x)}{t}
\end{align*}
since $B(x_t,R)\supseteq B(x_t,\rho(x,x_t))$ by the choice of $R$.

\bigskip

We are left to show the estimate $\gtrsim$, and it suffices to show that there exists a set $E$ with $\mu(E)>0$ such that for every $x\in E$ and some ball $B\ni x$ we have that
\[\frac{1}{\mu(B)}\int_{B}u_t\,d\mu\gtrsim \frac{u_t(x)}{t}.\]
To see this, recall that $x_t$ is a $(2C_\mu)^{-3/t}$-point so that there exists a ball $B=B(x_t,R)$ so that $\mu(B(x_t,R))>(2C_\mu)^{3/t}\mu(\{x_t\})$. Let $k\geq 1$ be the first integer such that $\mu(B(x_t,2^{-k}R))< 2^{-1}\mu(B)$, and set $r:=2^{-k}R$. Then the ball $B(x_t,2r)$ satisfies the inverse estimate, i.e.
\begin{equation}\label{est:estimate5}
\mu(B)\leq 2\mu(B(x_t,2r))\leq 2C_\mu \mu(B(x_t,r)).
\end{equation}
Set $E:=B\setminus B(x_t,r)$. Then $\mu(E)=\mu(B)-\mu(B(x_t,r))>2^{-1}\mu(B)>0$, and for every $x\in E$, the ball $B\ni x$ satisfies, by Lemma~\ref{lem:integraloverB0} and \eqref{est:estimate5}, the estimate
\[\frac{1}{\mu(B)}\int_{B}u_t\,d\mu 
\gtrsim \frac{1}{t\mu(B)^{1-t}}
\gtrsim \frac{1}{t\mu(B(x_t,r))^{1-t}}
\geq \frac{1}{t\mu(B(x_t,\rho(x,x_t)))^{1-t}}= \frac{u_t(x)}{t}.\]
\end{proof}

\begin{lemma}\label{lem:Lpnorm}
Let $0<t<1$ and suppose that $x_t$ is an $\varepsilon$-point with $\varepsilon =(2C_\mu)^{-3/t}$, and let $B=B(x_t,R)$ be a ball so that $\mu(B)>\varepsilon^{-1}\mu(\{x_t\})$. Then, for the function $f_t=\chi_{B}$, 
\[\|f_t\|_{L^p(u_t)} = u_t(B)^{1/p}\approx \left(\frac{\mu(B)^t}{t}\right)^{1/p}. \]
\end{lemma}
\begin{proof}
This is a special case of Lemma~\ref{lem:integraloverB0}.
\end{proof}

The proof of Lemma~\ref{lem:reduction2} is finally completed by the following lemma.

\begin{lemma}\label{lem:}
Given $0<t<1$, suppose $x_t$ is an $\varepsilon$-point with $\varepsilon=(2C_\mu)^{-2}(2C_\mu)^{-4/(t\gamma)}$, and let $B=B(x_t,R)$ be a ball so that $\mu(B)>\varepsilon^{-1}\mu(\{x_t\})$. Let $u_t$ be a power weight defined in \ref{def:powerweight} and set $f_t=\chi_{B}$. Then 
\[\|T_\gamma(f_tu_t^\gamma)\|_{L^{q,\infty}(u_t)}\gtrsim [u_t]_{A_1}^{1-\gamma}\|f_t\|_{L^p(u_t)}.\]
\end{lemma}

\begin{proof}
We pick a decreasing sequence $(r_k)$ of radii as in the proof of Lemma~\ref{lem:integraloverB0}: Let $r_0:=R$. Then let $k_1\geq 1$ be the smallest integer such that $\mu(B(x_t,2^{-k_1}r_0))<2^{-1}\mu(B(x_t,r_0))$, and set $r_1:=2^{-k_1}r_0$. Having chosen $k_m$ and $r_m$ in this fashion, let $k_{m+1}>k_m$ be the smallest integer such that $\mu(B(x_t,2^{-k_{m+1}}r_0))< 2^{-1}\mu(B(x_t,r_m))$, and set $r_{m+1}:= 2^{-k_{m+1}}r_0$. Again, if $\mu(\{x_t\})=0$, we may keep sub-dividing infinitely many times, and otherwise, we stop at the step $K\geq 1$ for which $\mu(B(x_t,r_{K-1}))> 2\mu(\{x_t\})$ and $\mu(B(x_t,r_K))\leq 2\mu(\{x_t\})$. 

We have the estimate 
\begin{equation}\label{eq:anestimate7}
\mu(B(x_t,r_{i}))\leq 2C_\mu\mu(B(x_t,r_{i+1}));
\end{equation} 
cf. \eqref{eq:anestimate2} (with $r_0=R$ in stead of $r_0=r$). Let $K\geq 1$ be an integer such that 
\[K-1<\frac{1}{t\gamma}\leq K. \]
Let us show that the sub-division into smaller balls by reducing to half the mass proceeds at least $K$ times. Indeed, by iterating \eqref{eq:anestimate7} and by the choice of $K$, we see that
\begin{align}\label{est:measure_middleball}
\mu(B(x_t,r_{K}))& \geq  (2C_\mu)^{-K}\mu(B(x_t,r_{0}))
\geq (2C_\mu)^{-1}(2C_\mu)^{-1/(t\gamma)}\mu(B).
\end{align}
First, \eqref{est:measure_middleball} implies that
\begin{equation}\label{est:muBepsilon1}
\mu(B(x_t,r_{K})) > 2C_\mu(2C_\mu)^{3/(t\gamma)}\mu(\{x_t\}).
\end{equation}
In particular, $\mu(B(x_t,r_{K}))>2\mu(\{x_t\})$ and thus, we may sub-divide $K$ times, as claimed. Set $r_t:=2^{-1}r_K$. Then $\mu(B(x_t,r_t))=\mu(B(x_t,2^{-1}r_K))\geq (C_\mu)^{-1}\mu(B(x_t,r_K))$, and by \eqref{est:muBepsilon1} we have that
\begin{equation}\label{est:muBepsilon2}
\mu(B(x_t,r_t))\geq (2C_\mu)^{3/(t\gamma)} \mu(\{x_t\})> (2C_\mu)^{3/t} \mu(\{x_t\}).
\end{equation}
Second, \eqref{est:measure_middleball} implies that
\begin{equation}\label{est:muBepsilon3}
\mu(B(x_t,r_t))^t\geq C \mu(B)^t, \quad C=C(\mu, \gamma).
\end{equation}

We then estimate
\begin{align*}
\|T_\gamma(f_tu_t^\gamma)\|_{L^{q,\infty}(u_t)} &=
\sup_{\lambda >0}\lambda \, u_t\left(\{ x\in X\colon T_\gamma (f_tu_t^\gamma)(x)>\lambda \}\right)^{1/q}\\
& \geq 
\sup_{\lambda >0}\lambda \, u_t\left(\{ x\in B(x_t,r_t)\colon T_\gamma (f_tu_t^\gamma)(x)>\lambda \}\right)^{1/q}.
\end{align*}
Let $x\in B(x_t,r_t)$. Then
\begin{align*}
T_\gamma (f_tu_t^\gamma)(x) & \geq \int_{X\setminus \{x\}}\frac{\chi_{B}(y)u_t^\gamma(y)d\mu(y)}{\mu(B(x,\rho(x,y)))^{1-\gamma}}
\geq \int_{B(x_t,R)\setminus B(x_t,2\rho(x_t,x))}\frac{u_t^\gamma(y)d\mu(y)}{\mu(B(x,\rho(x,y)))^{1-\gamma}}.
\end{align*}
Observe that for $y\notin B(x_t,2\rho(x,x_t))$, we have $B(x,\rho(x,y))\subseteq B(x_t,2A_0^2\rho(x_t,y))$, and the doubling property implies that $\mu(B(x,\rho(x,y)))\leq \mu(B(x_t,2A_0^2\rho(x_t,y)))\lesssim \mu(B(x_t,\rho(x_t,y)))$. Thus, 
\begin{align*}
T_\gamma (f_tu_t^\gamma)(x) & \gtrsim 
\int_{B(x_t,R)\setminus B(x_t,2\rho(x_t,x))}\frac{u_t^\gamma(y)d\mu(y)}{\mu(B(x_t,\rho(x_t,y)))^{1-\gamma}}\\
& = \int_{B(x_t,R)\setminus B(x_t,2\rho(x_t,x))}\frac{d\mu(y)}{\mu(B(x_t,\rho(x_t,y)))^{1-t\gamma}}\\
&\geq \sum_{i=0}^{K-1}
\int_{B(x_t,r_{i})\setminus B(x_t,r_{i+1})}\frac{d\mu(y)}{\mu(B(x_t,\rho(x_t,y)))^{1-t\gamma}}
\end{align*}
since $x\in B(x_t,r_t)$ and thereby, $2\rho(x_t,x)<2r_t =r_K$ by the choice of $r_t$ so that $B(x_t,2\rho(x_t,x))\subseteq B(x_t,r_K)$. From now on the estimates are very similar to the ones performed when proving the estimate $\gtrsim$ of Lemma~\ref{lem:integraloverB0} with the only deviation that here the exponent of the quantities $\mu(B(x_t,r_i))$ is $t\gamma$ in place of $t$. We may conclude with
\begin{align*}
T_\gamma (f_tu_t^\gamma)(x) & \gtrsim 
\mu(B(x_t,R))^{t\gamma}\frac{1-(2C_\mu)^{-Kt\gamma }}{1-(2C_\mu)^{-t\gamma}}.
\end{align*}
Recall from the beginning of the proof that $K$ is chosen to satisfy $K\geq 1/(t\gamma)$. Thus, $(2C_\mu)^{-Kt\gamma}<(2C_\mu)^{-1}<1/2$, so that
\[T_\gamma (f_tu_t^\gamma)(x)  \gtrsim  \frac{\mu(B(x_t,R))^{t\gamma}}{1-(2C_\mu)^{-t\gamma}}\gtrsim \frac{\mu(B)^{t\gamma}}{t\gamma}\quad\text{for $0<t\gamma<1$ and $x\in B(x_t,r_t)$}.\]
We have shown that
\begin{align*}
\|T_\gamma(f_tu_t^\gamma)\|_{L^{q,\infty}(u_t)}& \gtrsim 
\sup_{\lambda >0}\lambda \, u_t\left(\left\{ x\in B(x_t,r_t)\colon T_\gamma (f_tu_t^\gamma)(x)>\lambda \right\}\right)^{1/q}\\
& \gtrsim  \frac{\mu(B)^{t\gamma}}{2t\gamma} \, u_t\left(\left\{ x\in B(x_t,r_t)\colon T_\gamma (f_tu_t^\gamma)(x)>\frac{\mu(B)^{t\gamma}}{2t\gamma} \right\}\right)^{1/q}\\
&=\frac{\mu(B)^{t\gamma}}{2t\gamma}  u_t(B(x_t,r_t))^{1/q}.
\end{align*}
By \eqref{est:muBepsilon2}, we may use Lemma~\ref{lem:integraloverB0} to estimate $u_t(B(x_t,r_t))$ from below. Recalling \eqref{est:muBepsilon3}, we see that 
\[u_t(B(x_t,r_t))\gtrsim \frac{\mu(B(x_t,r_t))^t}{t}\gtrsim  \frac{\mu(B)^t}{t}.\]
Thus,
\begin{align*}
\|T_\gamma(f_tu_t^\gamma)\|_{L^{q,\infty}(u_t)}& \geq C
\frac{1}{t} \left(\frac{1}{t}\right)^{1/q} \mu(B)^{t\gamma}\mu(B)^{t/q} \\
&  \approx C\left(\frac{1}{t}\right)^{1-\gamma}\left(\frac{1}{t}\right)^{1/p}
\mu(B)^{t/p}=\left(\frac{1}{t}\right)^{1-\gamma}\left(\frac{\mu(B)^{t}}{t}\right)^{1/p},\quad C=C(A_0,C_\mu,\gamma),
\end{align*}
where we used the identity $1/p-1/q=\gamma$. Lemma~\ref{lem:A1} and Lemma~\ref{lem:Lpnorm} now complete the proof.
\end{proof}

\def\cprime{$'$} \def\cprime{$'$} \def\cprime{$'$}

\end{document}